\documentclass[12pt,leqno]{amsart}
\usepackage{hyperref,color}
\usepackage[T1]{fontenc}
\usepackage{amssymb}
\usepackage{amsthm,amsmath}
\numberwithin{equation}{section}
\newtheorem{thm}{Theorem}
\newtheorem{prop}{Proposition}

\newtheorem{lem}{Lemma}
\newtheorem{exmp}{Example}
\newtheorem{rem}{Remark}

\newcommand{\X}{E}

\newcommand{\supp}{\mathrm{supp}}
\newcommand{\B}{\mathcal{B}}

\newcommand{\E}{\mathcal{E}}
\renewcommand{\Sigma}{\E}

\newcommand{\Lloc}{L^1_{\mathrm{loc}}}
\newcommand{\ACloc}{AC}

\newcommand{\rg}{\mathrm{Im}}

\newcommand{\SX}{L^1}
\newcommand{\SXP}{L^1_{\varphi}}

\newcommand{\citep}{\cite}

\begin{document}

\title[Dynamics and density evolution in growth processes]
{Dynamics and density evolution in piecewise deterministic
growth processes}
\author{Michael C. Mackey}
\address{Departments of Physiology, Physics \& Mathematics and Centre for Nonlinear
Dynamics, McGill University, 3655 Promenade Sir William Osler, Montreal, QC,
Canada, H3G 1Y6} \email{ mackey@cnd.mcgill.ca}
\author{Marta Tyran-Kami\'nska${}^\dagger$}
\address{Institute of
Mathematics, Polish Academy of Sciences, Bankowa 14, 40-007 Katowice, Poland}
\address{Institute of Mathematics,
University of Silesia, Bankowa 14, 40-007 Katowice, Poland}
\email{mtyran@us.edu.pl}
\thanks{${}^\dagger$Corresponding author}


\dedicatory{Dedicated to the memory of Andrzej Lasota (1932-2006)}%

\begin{abstract}
A new sufficient condition is proved for the existence of stochastic semigroups
generated by the sum of two unbounded operators. It is applied to one-dimensional
piecewise deterministic Markov processes, where we also discuss the existence of a
unique stationary
density and give sufficient conditions for asymptotic stability.
\end{abstract}

\maketitle

\section{Introduction} \label{sec:intro}

The development of  cell cycle models to  account for the statistical properties
of division dynamics in populations of cells inevitably led to the consideration
of stochastically perturbed dynamical systems
\citep{diekmann84,almcm84,al-mcm-jt-92,tyrcha88,tyson86}. These applied
considerations have been followed by work on  the behaviour of Poisson driven
dynamical systems in a pure mathematical context \citep{lasota03,traple96}.  More
recently other areas of application related to the role of intrinsic (as opposed
to extrinsic) noise in gene regulatory dynamics
\citep{gillespie77,lipniacki06,rudnicki07} have made the understanding of
stochastic perturbations of dynamical systems of more than passing interest.

We were originally motivated by the work of Lasota et al.~\cite{al-mcm-jt-92} who
considered a (biological) system which produces `events' and has an internal or
physiological time in addition to the laboratory time $t$. We denote this internal
time by $\tau$ to distinguish from the time $t$. When an event appears the
physiological time $\tau=\tau_{e}$ is reset to $\tau=0$. We assume that the rate
${d\tau}/{dt}$ depends on the amount of an `activator' which we denote by $a$.
Thus we have
\begin{equation}\label{e:dephi}
\frac{d \tau}{dt}=\varphi(a),\quad \varphi\ge 0.
\end{equation}
The activator is produced by a dynamics described by the differential equation
\begin{equation}\label{e:de}
\frac{d a}{dt}=g(a),
\end{equation}
where $g\ge 0$ is a continuous function
on an open interval that may or may not be bounded. 
When an event is produced at a time $\tau_{e}$ and activator level $a_e$, then a
portion $\rho(a_e)$ of $a_e$ is consumed so the level of the activator after the
event is then
\begin{equation}\label{d:w}
\sigma(a_e):=a_e-\rho(a_e).
\end{equation}
We also allow  the possibility that the portion $\rho(a_e)$ depends on an
environmental or external factor so that  $\rho$ is a function of two variables
$\rho(a_e,\theta_e)$ where $\theta_e\in \Theta$ is distributed according to some
probability measure $\nu$ on $\Theta$. The solution of \eqref{e:de} with the
initial condition $a(0)=x$ will be denoted by
\[
a(t)=\pi_tx
\]
and we assume that it is defined for all $t\ge 0$. Then the solution of equation
\eqref{e:dephi} with the initial condition $\tau(0)=0$ is given by
\[
\tau(t)=\int_0^t\varphi(\pi_rx)dr.
\]
It is reasonable to require that also $\tau(t)$ is finite for all $t\ge 0$.

Lasota et al.~\cite{al-mcm-jt-92} studied the statistical behavior of a sequence
of such events occurring at random times
$$
0=t_0 < t_1 < \cdots < t_n < \cdots
$$
and denoted $a_n = a(t_n)$
to find 
\begin{equation}\label{eq:an}
a_{n+1} = T(a_n,\tau_n),\quad \text{where}\quad \tau_n = \int_{t_{n-1}}^{t_n}
\varphi(\pi_{s-{t_{n-1}}} a_{n-1}) ds
\end{equation}
were exponentially distributed independent random variables,  giving a relation
between successive activator levels at event occurrence and studying a discrete
time system with stochastic perturbations by the $\tau_n$. Here we have extended
these considerations to a continuous time situation by examining what happens at
all times $t$ and not merely what happens at $t_0,t_1,t_2,  \cdots$. Thus we
arrive at a continuous time piecewise deterministic Markov process, whose sample
paths between the jump times $t_0, t_1, \cdots$ are given by the solution of
\eqref{e:de} and at the jump times the state of the process is
selected according to a jump stochastic kernel, 
which is the transition probability function for~\eqref{d:w}. This leads us to
study evolution equations of the form
\begin{equation}\label{eq:evd2}
\dfrac {\partial u }{\partial t}=A_0u-\varphi u+P(\varphi u), \quad
\text{where}\quad  A_0u(x)=-\frac{d}{dx}(g(x)u(x)),
\end{equation}
on the space of integrable functions $L^1$, where $P$  is a stochastic operator on
$L^1$ corresponding to the jump stochastic kernel and $\varphi$  need not be
bounded. We supplement  \eqref{eq:evd2} with the initial condition $u(0)=u_0$
which is the density of the distribution of the initial amount of the activator.

Let us write
\begin{equation}\label{eq:opAC}
Au=A_0u-\varphi u\quad \text{and}\quad  \mathcal{C}u=Au+P(\varphi u).
\end{equation}
If $\varphi$ is unbounded then $\mathcal{C}$ is the sum of two unbounded
operators, so the existence and uniqueness of solutions to the Cauchy problem in
$\SX$ is problematic; \eqref{eq:evd2} may have multiple
solutions~\cite{banasiakarlotti06}. We make use of perturbation results for
positive semigroups on $L^1$-spaces which go back to~\cite{kato54} (see
Section~\ref{sec:Pert}),  from which it follows that the operator $\mathcal{C}$
has an extension $C$ generating a positive contraction semigroup $\{P(t)\}_{t\ge
0}$ provided that the operator $A$ is the infinitesimal generator of a positive
contraction semigroup on $\SX$ and $\mathcal{C}$ is defined on the domain of $A$.
In general, if the closure of $\mathcal{C}$ is the generator $C$ of
$\{P(t)\}_{t\ge 0}$ then the Cauchy problem is uniquely solved and $\{P(t)\}_{t\ge
0}$ is a stochastic semigroup. In Section~\ref{sec:Pert} we prove a new sufficient
condition for uniqueness and in Section~\ref{sec:Pt} we show that if the discrete
process has a strictly positive stationary density then uniqueness holds. This
simplifies the analysis of \eqref{eq:evd2} when compared with the approach in
\cite{banasiakarlotti06}, and allows us to investigate  both the uniqueness of
solutions and their asymptotic properties.


The outline of this paper is as follows.  We recall basic definitions and
fundamental theorems from the theory of stochastic operators and semigroups in
Section~\ref{ssec:SDAS} and perturbation results for positive semigroups on
$L^1$-spaces in Section~\ref{sec:Pert}, which closes with the proof of our main
general result (Theorem~\ref{thm:uniq}). In Section~\ref{sec:St} we prove that the
operators $A_0$ and $A$ defined on suitable domains are generators. In
Section~\ref{sec:Pt} we show the applicability of Theorem~\ref{thm:uniq} to
equation~\eqref{eq:evd2} when $P$ is an arbitrary stochastic operator and also
give sufficient conditions for asymptotic stability. In Section
\ref{sec:exp-growth} we let the operator $P$ have a definite form that fits
directly into our framework,  and give several concrete examples drawn from work
on the regulation of the cell cycle as well as classical integro-differential
equations. In Section~\ref{sec:other} we extend our results to the situation in
which there is degradation (as opposed to growth)  and illustrate their
applicability using models for the stochastic regulation of gene expression.

In a companion paper \cite{tyran07}, these and other  results are placed in the
general context of semigroup theory and probability theory with  applications to
piecewise deterministic Markov process without `active boundaries'. There we also
prove that when the semigroup $\{P(t)\}_{t\ge 0}$ is stochastic then
\eqref{eq:evd2} is the corresponding evolution equation for densities of such
processes.

\section{Stochastic operators and semigroups}
\label{ssec:SDAS}

Let $(\X,\Sigma,m)$ be a $\sigma$-finite measure space. We denote by $D$ the set
of all densities on $\X$, i.e.
$$
D=\{u\in \, \SX: \,\, u\ge 0,\,\, \|u\|=1\},
$$
where $\|\cdot\|$ is the norm in $\SX=L^1(\X,\Sigma,m)$. A linear operator
$P\colon \SX\to \SX$ such that $P(D)\subset D$ is called \emph{stochastic} or
\emph{Markov}~\cite{almcmbk94}.

Let $P\colon \SX\to \SX$ be a stochastic operator. A density
$u$ is said to be \emph{invariant} or \emph{stationary} for $P$
if $Pu=u$. We say that $P$ \textit{overlaps supports} if for
every $u, v\in D$ there is a positive integer $n\ge 1$ such
that
$$
m (\supp P^{n}u\cap \supp P^{n}v)>0,
$$
where the support of  $u\in \SX$ is defined up to a set of measure zero by the
formula $ \supp\, u = \{x\in \X: u(x)\neq 0\}. $ Note that if $P$ overlaps
supports then it can have at most one invariant density \cite[see the proof of
Corollary~1]{rudnicki95}.

Let $\mathcal{J}\colon \X\times \E\to[0,1]$ be a \emph{stochastic transition
kernel}, i.e. $\mathcal{J}(x,\cdot)$ is a probability measure for each $x\in \X$
and the function $x\mapsto\mathcal{J}(x, B)$ is measurable for each $B\in\E$,
and let $P$ be a stochastic operator on $\SX$. If 
\begin{equation}\label{eq:stkmo}
\int_\X \mathcal{J}(x,B)u(x)m(dx)=\int_B Pu(x)m(dx)\quad \text{for all } B\in \E,
u\in D(m),
\end{equation}
then $P$ is called the \emph{transition} operator corresponding to $\mathcal{J}$.

If  $\mathcal{J}(x,B)=1_{T^{-1}(B)}(x)$ for $x\in \X$, $B\in\E$, where $T\colon
\X\to \X$  is a \emph{nonsingular} measurable transformation, i.e.
$m(T^{-1}(B))=0$ for all $B\in\E$ such that $m(B)=0$, then there exists a unique
stochastic operator $P$ on $\SX$ satisfying \eqref{eq:stkmo} and $P$ defined by
\eqref{eq:stkmo} is called the \emph{Frobenius-Perron} operator corresponding to
$T$.

A stochastic operator $P$ on $\SX$ is called \emph{partially integral} or
\emph{partially kernel} if there exists a measurable function $p\colon \X\times
\X\to[0,\infty)$ such that
$$
\int_\X\int_\X p(x,y)\,m(dy)\, m(dx) >0 \quad\text{and}\quad P u(x) \ge \int_\X
p(x,y)u(y)\, m(dy)
$$
for every density $u$. If, additionally,
\[
\int_\X p(x,y)\,m(dx)=1,\quad  y\in \X,
\]
then $P$ corresponds to the stochastic kernel
\[
\mathcal{J}(x,B)=\int_{B}p(y,x)\,m(dy),\quad x\in \X, B\in \E\] and we simply say
that $P$ has \emph{kernel} $p$.

A strongly continuous semigroup $\{P(t)\}_{t\ge 0}$ on $\SX$ is called a
\emph{stochastic semigroup} or \emph{Markov semigroup}  if  $P(t)$ is a stochastic
operator for all $t\ge 0$. A density $u_*$ is called \emph{invariant} or
\emph{stationary} for $\{P(t)\}_{t\ge 0}$ if $P(t) u_*=u_*$ for every $t\ge 0$.

A stochastic semigroup $\{P(t)\}_{t\ge 0}$ is called \emph{asymptotically stable}
if there is a stationary density $u_*$  such that
$$
\lim_{t\to\infty}\|P(t) u-u_*\|=0\;\;\text{for}\;\;u\in D
$$
and it is called \emph{partially integral} if, for some
$t_0>0$, the operator $P(t_0)$ is partially integral.

\begin{thm}[\citep{pichorrudnicki00}]\label{thm:pr00}
Let  $\{P(t)\}_{t\ge 0}$ be a partially integral stochastic semigroup. Assume that
the semigroup $\{P(t)\}_{t\ge 0}$ has only one invariant density $u_*$. If $u_*>0$
a.e. then the semigroup $\{P(t)\}_{t\ge 0}$ is asymptotically stable.
\end{thm}

\section{Perturbation results in $\SX$}\label{sec:Pert}

Let $(A_0,\mathcal{D}(A_0))$ be the infinitesimal generator of a stochastic
semigroup, $\varphi\ge 0$ be a measurable function, and
\[
\SXP=\{u\in \SX:\int_\X\varphi(x)|u(x)|m(dx)<\infty\}.
\]
Let  $P$ be a stochastic operator on $\SX$ and let the operators $A$ and
$\mathcal{C}$, as given in \eqref{eq:opAC}, be defined on $\mathcal{D}(A)\subseteq
\mathcal{D}(A_0)\cap \SXP$. Assume that the operator $(A,\mathcal{D}(A))$ is the
infinitesimal generator of a positive strongly continuous contraction semigroup
$\{S(t)\}_{t\ge 0}$ on $\SX$. Then it is known
\cite{kato54,voigt87,banasiak01,banasiakarlotti06} that there exists a positive
strongly continuous contraction semigroup $\{P(t)\}_{t\ge 0}$ on $\SX$ satisfying
the following:
\begin{enumerate}
\item the infinitesimal generator $C$ of $\{P(t)\}_{t\ge 0}$ is an extension of the operator
$\mathcal{C}$, i.e. $\mathcal{D}(A)\subseteq\mathcal{D}(C)$ and $Cu=\mathcal{C}u$
for $u\in\mathcal{D}(A)$;
\item if $\{\bar{P}(t)\}_{t\ge 0}$ is another semigroup generated by an extension of $\mathcal{C}$
then $\bar{P}(t)u\ge P(t)u$ for all $u\in\SX$, $u\ge 0$, i.e.~$\{P(t)\}_{t\ge 0}$
is the minimal semigroup;
\item the generator $C$ is characterized by
\begin{equation}\label{eq:rp}
R(\lambda,C)u=\lim_{N\to\infty}R(\lambda,A)\sum_{n=0}^N (P(\varphi
R(\lambda,A)))^nu, \quad u\in \SX, \lambda>0,
\end{equation}
where $R(\lambda,\cdot)$ is the resolvent operator;
\item the semigroup $\{P(t)\}_{t\ge 0}$ satisfies the equation
\begin{equation}\label{eq:df}
P(t)u=S(t)u+\int_{0}^t P(t-s) P(\varphi S(s)u)\,ds, \quad u\in \mathcal{D}(A).
\end{equation}
\end{enumerate}

We can not conclude, in general, that the semigroup $\{P(t)\}_{t\ge 0}$ is
stochastic \cite[Example 4.3]{kato54}. 
Discussing various conditions for this to hold has
been a major objective of study \cite{kato54,voigt87,banasiak01,frosali04,
banasiakarlotti06} and leads to the following result.
\begin{thm}
\label{thm:pert} If for some $\lambda>0$
\begin{equation}\label{eq:css}
\lim_{n\to\infty}\|(P(\varphi R(\lambda,A)))^nu\|=0\quad\text{for all } u\in \SX
\end{equation}
then $\{P(t)\}_{t\ge 0}$ is a stochastic semigroup and its generator $C$ is the
closure of the operator $(\mathcal{C},\mathcal{D}(A))$.
\end{thm}

We now prove our main result.
\begin{thm}\label{thm:uniq}
If for some $\lambda>0$ there is $v\in \SX$ such that $v>0$ a.e. and  $P(\varphi
R(\lambda,A))v\le v$, then condition \eqref{eq:css} holds.
\end{thm}
\begin{proof}
We can assume that $v$ is a density.  Set $K_\lambda=P(\varphi R(\lambda,A))$. The
operator $K_\lambda$ is a positive contraction.  Since $K_\lambda v\le v$, the
sequence $K_\lambda^nv$ is strongly convergent in $\SX$ to a fixed point $u_*$ of
the operator $K_\lambda$. From \cite[Theorem
 4.3]{banasiakarlotti06} it follows that $u_*=0$. For any density $u$ we have
\[
0\le K_\lambda^n u_k\le k K_\lambda^n v, \quad \text{where}\quad u_k=\min\{u,k
v\},\; k\ge 1, n\ge1.
\]
Since $\|u_k-u\|\to 0$ as $k\to\infty$, this completes the proof of
\eqref{eq:css}.
\end{proof}

\section{The semigroup $\{S(t)\}_{t\ge 0}$}\label{sec:St}

Let $\X$ be an open interval in $\mathbb{R}$, bounded or unbounded,
$\Sigma=\B(\X)$ and $m$ be the Lebesgue measure. We shall denote by $\Lloc$  the
space of all Borel measurable functions on $\X$ which are integrable on compact
subsets of $\X$ and by $\ACloc$ the space of all absolutely continuous functions
on $\X$. We assume from now on that $\X=(d_0,d_1)$, where $-\infty\le d_0<d_1\le
\infty$, $g\colon \X\to \mathbb{R}$ is a continuous strictly positive function,
and $\varphi\in \Lloc$ is nonnegative. In this section we study the first order
differential operators $A_0$ and $A$ which are meaningful for any function $u\in
\Lloc$ for which $gu\in \ACloc$. We will define the operators $A_0$ and $A$ on
suitable domains $\mathcal{D}(A_0)$ and $\mathcal{D}(A)$ so that they are
generators of corresponding semigroups as described in Section~\ref{sec:Pert}.

Since $1/g,\varphi/g\in \Lloc$, we can define
\begin{equation}\label{eq:GQ}
G(x)=\int_{x_0}^x \frac{1}{g(z)}dz\quad \text{and }\quad
Q(x)=\int_{x_1}^x\frac{\varphi(z)}{g(z)}dz,
\end{equation}
where $x_0=d_0$ and $x_1=d_0$ when the integrals exist for all $x$ and, otherwise,
$x_0$, $x_1$ are any points in $\X$. The function $G$ is strictly monotonic,
continuously differentiable on $\X$, $G(d_0)\in\{0,-\infty\}$, and $G(d_1)$ is
either finite or equal to $+\infty$. The function $Q$ is monotonic with
$Q(d_0)\in\{0,-\infty\}$ and $Q(d_1)$ is either finite or equal to $+\infty$. $G$
is invertible with $G^{-1}$ well defined on $G(\X)$. If $G(\X)\neq \mathbb{R}$,
then we extend $G^{-1}$ continuously so that $G^{-1}(\mathbb{R}\setminus
G(\X))=\{d_0,d_1\}$. The formula
\begin{equation*}\label{d:rtx}
a(t,x)=G^{-1}(G(x)+t),\quad x\in \X, t\in\mathbb{R},
\end{equation*}
defines a  monotone continuous function in each variable with values in
$[-\infty,\infty]$ which is a solution of \eqref{e:de}. If $G(d_1)=+\infty$ then
$\pi_tx=a(t,x)$ is finite for all $t\ge 0$, $x\in \X$, and $\pi_t(\X)\subseteq
\X$, $t\ge 0$. In the case when $|G(d_0)|=G(d_1)=\infty$ the value $a(t,x)$ is
finite for all $t\in\mathbb{R}$ and $x\in \X$, so that we have, in fact, a flow
$\pi_t$ on $\X$ such that $\pi_t(\X)=\X$.

For $t>0$ we define the operators $P_0(t)$  and $S(t)$ on $\SX$ by
\begin{equation}\label{eq:dpo}
P_0(t)u(x)=\mathbf{1}_\X(\pi_{-t}x)u(\pi_{-t}x)\frac{g(\pi_{-t}x)}{g(x)},\;\;x\in
\X, 
u\in \SX
\end{equation}
and
\begin{equation}\label{eq:dst}
S(t)u(x)=e^{Q(\pi_{-t}x)-Q(x)}P_0(t)u(x),\quad  x\in \X, u\in \SX.
\end{equation}

\begin{thm}\label{thm:sP0t} If $G(d_1)=+\infty$ then $\{P_0(t)\}_{t\ge 0}$
is a stochastic semigroup 
and $\{S(t)\}_{t\ge 0}$ is a positive strongly continuous contraction semigroup on
$\SX$.
\end{thm}
\begin{proof}
$P_0(t)$ is a stochastic operator because it is the Frobenius-Perron operator for
the transformation $x\mapsto
\pi_tx$. 
Since $Q$ is nondecreasing and $t\mapsto\pi_t x$ is increasing, we always have
\[
e^{Q(\pi_{-t}x)-Q(x)}\mathbf{1}_{\X}(\pi_{-t}x)\le \mathbf{1}_{\X}(\pi_{-t}x).
\]
Hence $S(t)$ is a positive contraction.
 To check the semigroup property observe that if
$x\in \X$ and $\pi_{-s-t}x\in \X$, then $\pi_s(\pi_{-s-t}x)\in \X$, by assumption,
and thus $\pi_s(\pi_{-s-t}x)=\pi_{-t}x\in \X$.  Furthermore, if $x\in \X$ and
$\pi_{-t}x\in \X$ then $\pi_{-s}(\pi_{-t}x)=\pi_{-s-t}x$. Hence
\[
\mathbf{1}_\X(\pi_{-t}x)\mathbf{1}_\X(\pi_{-s}(\pi_{-t}x))=\mathbf{1}_\X(\pi_{-s-t}x),
\]
which shows that $S(t)S(s)u(x)=S(s+t)u(x)$ for $ t,s\ge 0, x\in \X, u\in \SX$.
Finally we must show that $\{S(t)\}_{t\ge 0}$ is strongly continuous.
Let $u\in C_c(\X)$, where $C_c(\X)$ is the space of continuous functions which are equal to zero near boundaries.  
For every $x\in \X$ and all sufficiently small $t>0$ we have
$\mathbf{1}_\X(\pi_{-t}x)=1$ and $\pi_{-t}x\to x$ as $t\downarrow 0$.
Consequently,
\[
\lim_{t\downarrow 0}S(t)u(x)=u(x)\quad x\in \X,
\]
which, by the Lebesgue dominated convergence theorem, implies
\[
\lim_{t\downarrow 0}\|S(t)u-u\|=0.
\]
Since the set $C_c(\X)$ is a dense subset of $\SX$, this shows the strong
continuity of the semigroup $\{S(t)\}_{t\ge0}$. If we take $\varphi\equiv 0$ then
$S(t)=P_0(t)$ and we recover the claim for $\{P_0(t)\}_{t\ge 0}$.
\end{proof}

Now we identify the generators of the semigroups from Theorem~\ref{thm:sP0t}. The
maximal domain of $A_0$ in $\SX$ consists of all functions $u\in \SX$ such that
$A_0u\in \SX$. Then the integrability of $A_0u$ implies existence of the finite
limits
\begin{equation}
l_0(u):=\lim_{x\downarrow d_0}g(x)u(x)\quad \text{and}\quad
l_1(u):=\lim_{x\uparrow d_1}g(x)u(x).
\end{equation}
A necessary condition for $A_0$ to generate a stochastic semigroup is that the
limits $l_i$ are equal. Recall that $u\in \SXP$ if and only if $u\in \SX$ and
$\varphi u\in \SX$.

\begin{thm}\label{thm:ASt} If  $G(d_1)=+\infty$
then the operator $A_0$ defined on the domain
\begin{equation}\label{d:dA0}
\mathcal{D}(A_0)=\{u\in \SX: gu\in \ACloc,\; A_0 u\in \SX,\; \lim_{x\downarrow
d_0}g(x)u(x)=0\}
\end{equation}
is the generator of the semigroup $\{P_0(t)\}_{t\ge 0}$ and the operator $A$
defined on $\mathcal{D}(A)=\mathcal{D}(A_0)\cap \SXP$  is the generator of
$\{S(t)\}_{t\ge 0}$.
\end{thm}

Before we give the proof of Theorem~\ref{thm:ASt} we first provide general
properties of the operator $A$. Let $\lambda>0$.
 Define the function
$r_\lambda\colon \X\times \X\to[0,\infty)$  by
\begin{equation}\label{eq:Rlamk}
r_\lambda(x,y)= \mathbf{1}_{(d_0,x)}(y)\frac{e^{Q_\lambda(y)-Q_\lambda(x)}}{g(x)}
,\quad \text{where}\quad Q_\lambda(z)=\lambda G(z)+Q(z),
\end{equation}
and the positive linear operator $R_\lambda \colon\SX\to \SX$ by
\begin{equation}\label{eq:Rlam}
R_\lambda v(x)=\int_{d_0}^{d_1}r_\lambda(x,y)v(y)\,dy.
\end{equation}

\begin{lem}\label{l:Rlam} Let $\lambda> 0$.
The operator $R_\lambda$ satisfies
\[
\lambda\| R_\lambda v\|\le \|v\|,\quad v\in \SX. 
\]
For every $v\in \SX$ we have $gR_\lambda v\in\ACloc$ and the function $u=R_\lambda
v$ is a particular solution in $\SX$ of the equation
\begin{equation}\label{eq:resA}
\lambda u-Au=v.
\end{equation}
\end{lem}
\begin{proof}
$R_\lambda$ is an integral operator with nonnegative measurable kernel
$r_\lambda$. Observe that
\[
\lambda r_\lambda(x,y)\le
\mathbf{1}_{(d_0,x)}(y)\frac{\lambda}{g(x)}e^{\lambda(G(y)-G(x))} \quad \text{for
} x,y\in \X.
\]
Thus for every $y\in \X$ we have
\[
\begin{split}
\lambda \int_{d_0}^{d_1}r_\lambda(x,y)\,dx\le  1 - e^{-\lambda (G(d_1)-G(y))}\le
1,
\end{split}
\]
because  $G(d_1)\ge G(y)$, showing $\lambda\| R_\lambda \|\le 1$. Since $e^{-Q}\in
\ACloc$, we can write
\[
Au(x)=-e^{-Q(x)}\frac{d}{dx}\Bigl(g(x)u(x)e^{Q(x)}\Bigr),
\]
for all $u$ such that $gu\in \ACloc$.  We have
\[
g(x)R_\lambda v(x)=e^{-Q_\lambda(x)}\int_{d_0}^{x}e^{Q_\lambda(y)}v(y)\,dy
\]
and $Q_\lambda(y)\le Q_\lambda(\alpha)$ for every $y\le \alpha<d_1$, thus the
function $e^{Q_\lambda}v$ is integrable on intervals $(d_0,\alpha]$ for every
$\alpha<d_1$. Hence, $gR_\lambda v\in \ACloc$ and $R_\lambda v$ satisfies
\eqref{eq:resA}.
\end{proof}

\begin{lem}\label{l:Rlam2}
If $Q_\lambda(d_1)=+\infty$ then $R_\lambda(\SX)\subseteq \mathcal{D}(A)$ and
\[
R_\lambda (\lambda u-Au)=u \quad \text{for} \quad u\in\mathcal{D}(A).
\]
\end{lem}
\begin{proof}
First we show that $R_\lambda(\SX)\subseteq\mathcal{D}(A)$.
 Let $v\in \SX$. Then
$gR_\lambda v\in\ACloc$,
\[
|g(x)R_\lambda v(x)|\le \int_{d_0}^{d_1}|g(x)r_\lambda(x,y)| |v(y)|\,dy
\quad\text{for}\quad x\in \X,
\]
and $|g(x)r_\lambda(x,y)|\le \mathbf{1}_{(d_0,x)}(y)$ for all $x,y\in \X$. From
the definition of $r_\lambda$ and the assumption $Q_\lambda(d_1)=+\infty$ it
follows that
\[
\lim_{x\to d_i} |g(x)r_\lambda(x,y)|=0 \quad \text{for } y\in \X.
\]
By the Lebesgue dominated convergence theorem,  $ l_i(R_\lambda v)=0$, $i=0,1$,
for all $v\in \SX$. Since $|v|\in \SX$, the function $R_\lambda |v|$ is a
particular solution of
\[
(\lambda+\varphi) R_\lambda |v|= |v|+A_0(R_\lambda |v|)
\]
and $l_i(R_\lambda |v|)=0$ for $i=0,1$. Hence
\begin{equation*}\label{e}
\int_{d_0}^{d_1}(\lambda+\varphi(x)) R_\lambda
|v|(x)\,dx=\int_{d_0}^{d_1}|v(x)|\,dx,
\end{equation*}
which shows that $\|\varphi R_\lambda |v|\|<\infty$ and $R_\lambda v\in \SXP$.
Finally, from \eqref{eq:resA} it follows that
\[
A_0(R_\lambda v)=(\lambda +\varphi)R_\lambda v-v\in \SX. \]

Now let $ u\in \mathcal{D}(A)$ and $v:=\lambda u-Au$. Since $u\in \SXP$ and
$A_0u\in \SX$, we have $v\in \SX$ and $R_\lambda v\in\mathcal{D}(A)$. Hence
$w:=u-R_\lambda v\in \mathcal{D}(A)$ and $ Aw=\lambda w. $ The general solution
$w$ of this equation is of the form
\[
w(x)=c \frac{e^{-Q_\lambda(x)}}{g(x)},\quad x\in \X,
\]
where $c$ is a constant. Thus
\[
g(x)(u(x)-R_\lambda v(x))=c e^{-Q_\lambda(x)}\quad \text{for }x\in \X.
\]
If $|Q(d_0)|=+\infty$ and $c\neq 0$, then $ce^{-Q_\lambda}(\lambda +\varphi)/g$ is
not integrable, which contradicts $u-R_\lambda v\in \mathcal{D}(A)$ and gives
$c=0$. If $Q_\lambda(d_0)=0$  then $\l_0(u-R_\lambda v)=c$, which also gives $c=0$
and shows that $u=R_\lambda v$.
\end{proof}

\begin{proof}[Proof of Theorem~\ref{thm:ASt}]
Since $G(d_1)=+\infty$, the assumptions of Lemma~\ref{l:Rlam2} hold even when
$\varphi=0$. Thus it is enough to show that $(A,\mathcal{D}(A))$ is the generator
of $\{S(t)\}_{t\ge 0}$. Observe that $\SX= \rg(\lambda-A)$, by
Lemma~\ref{l:Rlam2}. Since \eqref{eq:resA} has a unique solution $u=R_\lambda v$,
we have $R_\lambda=(\lambda-A)^{-1}$ for $\lambda>0$ and
\[
\|\lambda u-Au\|=\|v\|\ge \lambda \|R_\lambda v\|=\lambda \|u\|\quad \text{for }
u\in \mathcal{D}(A).
\]
The operator $\lambda R_\lambda$ is a positive contraction. By the Hille-Yosida
theorem, the operator $A$ is a generator of a positive contraction semigroup.  Let
$\mathcal{A}$ be the generator of the semigroup $\{S(t)\}_{t\ge 0}$. It remains to
prove that $\mathcal{A}=A$.

 First, we will show that 
\[
R(\lambda,\mathcal{A})v=R_\lambda v\quad \text{for } v\in C_c(\X), \;v\ge 0.
\]
Let $ v\in C_c(\X)$, $v\ge 0$. 
We have
\[
R(\lambda,\mathcal{A})v=\lim_{t\to \infty}\int_{0}^t e^{-\lambda s}S(s)v\,ds,
\]
where the integral is an element of $\SX$ such that
\[
\left (\int_{0}^t e^{-\lambda s}S(s)v\,ds\right )(x)=\int_{0}^t e^{-\lambda
s}S(s)v(x)\,ds.
\]
Let $x\in \X$ and $t>0$. Define $s_*(x)=\sup\{s>0: \pi_{-s}x\in
\X\}$ and note that $\pi_{-s_*(x)}x=d_0$.  
Making use of the formula for $S(t)v$ and the fact that $G(\pi_{-s}x)=G(x)-s$ when
$\pi_{-s}x\in \X$ leads to
\[
\begin{split}
\int_{0}^t e^{-\lambda s}S(s)v(x)\,ds&=\int_{0}^{t\wedge s_*(x)}
e^{Q_\lambda(\pi_{-s}x)-Q_\lambda(x)}\frac{v(\pi_{-s}x)g(\pi_{-s}x)}{g(x)}\,ds.
\end{split}
\]
By a change of variables, we obtain
\[
\left |\int_{0}^t e^{-\lambda s}S(s)v(x)\,ds-R_\lambda v(x) \right |\le
\frac{1}{g(x)}
e^{-Q_\lambda(x)} w(t,x), 
\]
where
\[
w(t,x)=\left\{\begin{array}{ll} \int_{d_0}^{\pi_{-t}x}e^{Q_\lambda(z)}v(z)\,dz, &
\hbox{if } t< s_*(x);\\ 0, & \hbox{if } t\ge s_*(x).\\
\end{array}
\right.
\]
We have $w(t,x)\downarrow 0$ as $t\uparrow \infty$, thus 
\[
\lim_{t\to\infty}\left \|\int_{0}^t e^{-\lambda s}S(s)v\,ds-R_\lambda v \right
\|=0.
\]

Since $C_c(\X)$ is a dense subset of $\SX$ and both operators $\lambda
R(\lambda,\mathcal{A})$ and $\lambda R_\lambda$ are positive contractions, they
are identical. This shows that $\mathcal{D}(\mathcal{A})=\mathcal{D}(A)$ and
$(\lambda -\mathcal{A})u=(\lambda-A)u$ for $u\in\mathcal{D}(A)$, which completes
the proof.
\end{proof}

\begin{rem}
Observe that if $G(d_0)=-\infty$ then the domain of $A_0$ is
\[
\mathcal{D}(A_0)=\{u\in \SX: gu\in \ACloc,\; A_0 u\in \SX\}.
\]
\end{rem}

\section{Asymptotic properties }\label{sec:Pt}

In this section we assume that $g>0$ and $G(d_1)=+\infty$. Let
$P\colon\SX\to\SX$ be a stochastic operator and let
$\mathcal{C}$ be the operator
\[
\mathcal{C}u(x)=-\frac{d}{dx}(g(x)u(x))-\varphi(x) u(x)+P(\varphi u)(x),\quad x\in
(d_0,d_1),
\]
defined on  $\mathcal{D}(A)=\mathcal{D}(A_0)\cap \SXP$ with
$\mathcal{D}(A_0)$ as in~\eqref{d:dA0}. By
Section~\ref{sec:Pert}, it follows from Theorems~\ref{thm:sP0t}
and~\ref{thm:ASt} that there is a positive, strongly
continuous, contraction semigroup $\{P(t)\}_{t\ge 0}$ on $\SX$
whose generator is an extension of the operator
$(\mathcal{C},\mathcal{D}(A))$. In this section we give
sufficient conditions for $\{P(t)\}_{t\ge 0}$ to be a
stochastic semigroup and study its asymptotic properties.

Define the operator $R_0$ on $\mathcal{D}(R_0)=\{v\in\SX: R_0v\in \SX\}$  by
\begin{equation*}
R_0 v(x)=\int_{d_0}^{d_1} r_0(x,y)v(y)\,dy,\quad \text{where}\quad
r_0(x,y)=\mathbf{1}_{(d_0,x)}(y)\frac{1}{g(x)} e^{Q(y)-Q(x)}.
\end{equation*}
In general, $R_0$ may be an unbounded operator. 
We have
\[
\varphi(x)R_0v(x)=\int_{d_0}^{d_1}\varphi(x)r_0(x,y)v(y)\,dy
\]
for every $v\in\mathcal{D}(R_0)$. Since $Q$ is nondecreasing, 
\[
\int_{d_0}^{d_1}\varphi(x)r_0(x,y)\,dx=\int_{y}^{d_1}\frac{\varphi(x)}{g(x)}
e^{Q(y)-Q(x)}dx=1-e^{Q(y)-Q(d_1)}\le 1
\]
for every $y\in \X$. Thus
the operator $\varphi R_0$ can be uniquely extended, with the same formula, to a
positive contraction on $\SX$. Observe that $\varphi R_0$ is  stochastic  if and
only if $Q(d_1)=+\infty$.

Define the operator $K$  by
\begin{equation}\label{d:opP}
Ku=P(\varphi R_0) u\quad\text{for } u\in \SX.
\end{equation}
Since $P$ is a stochastic operator, $K$ is  stochastic  if and only if
$Q(d_1)=+\infty$.

%
%
%

\begin{thm}\label{thm:uniqa}
Assume that $Q(d_1)=+\infty$.

If the operator $K$ has an invariant density $v_*>0$ a.e. then $\{P(t)\}_{t\ge 0}$
is a stochastic semigroup. Moreover, if $\{P(t)\}_{t\ge 0}$ is partially integral
and $R_0v_*\in \SX$ then $\{P(t)\}_{t\ge 0}$ is asymptotically stable with
invariant density $u_*=R_0v_*/\|R_0v_*\|$.

Conversely, if  $\{P(t)\}_{t\ge 0}$ has an invariant density $u_*\in
\mathcal{D}(A)$ then the density $P(\varphi u_*)/\|\varphi u_*\|$ is invariant for
the operator $K$.
\end{thm}
\begin{proof} The operator $R_1=R(1,A)$ as defined in \eqref{eq:Rlam}
satisfies $(\varphi R_1)u\le (\varphi R_0)u$ for every density $u$, which is due
to the fact that $\varphi\ge 0$ and $r_1\le r_0$. Since $P$ is a
positive operator, 
we have $P(\varphi R(1,A))v_*\le v_*$ and $\{P(t)\}_{t\ge 0}$ is a stochastic
semigroup, by Theorems~\ref{thm:uniq} and~\ref{thm:pert}. Now let $C$ be the
generator of $\{P(t)\}_{t\ge 0}$. From \eqref{eq:rp} it follows that $R(1,C)u\ge
R_1 u$ for every density $u$. Since there is a $b(u)\in \X$ such that $R_1u(x)>0$
for $x\ge b(u)$, the stochastic operator $R(1,C)u$ overlaps supports. Recall that
$u_*$ is a stationary density for the semigroup $\{P(t)\}_{t\ge 0}$ if and only if
$u_*\in\mathcal{D}(C)$ and $Cu_*=0$. We have
\[v_*=P(\varphi
R_0)v_*=\|R_0v_*\|P(\varphi u_*).\] It is easily seen that $u_*\in\mathcal{D}(A)$
and $Au_*=-v_*/\|R_0v_*\|$. Hence $Cu_*=0$ and Theorem~\ref{thm:pr00} applies.

Finally, suppose that $P(t)u_*=u_*$ for all $t\ge 0$ with $u_*\in\mathcal{D}(A)$.
Since $\mathcal{D}(A)\subset \SXP$ and $Cu_*=Au_*+P(\varphi u_*)$, we
obtain $-Au_*=P(\varphi u_*)$. Thus $v:=P(\varphi u_*)\in \SX$  and 
$ Kv=P(\varphi R_0 v)=-P(\varphi R_0Au_*)$. It is easily seen that $R_0Au_*=-u_*$
and we recover the claim.
\end{proof}

\begin{rem}
Observe that when $\varphi(x)\ge b>0$ for all $x$, then the operator $R_0$ is
bounded, thus $R_0=R(0,A)=-A^{-1}$ and $R_0v_*$ is integrable. Moreover, if
$\varphi$ is a constant function, $\varphi\equiv b$, then
$(\mathcal{C},\mathcal{D}(A))$ is the generator of $\{P(t)\}_{t\ge 0}$, by the
Phillips perturbation theorem,  and $K=P(b R(b,A_0))$. In that case, the relation
between the invariant densities for the operator $K$ and the semigroup
$\{P(t)\}_{t\ge0}$ are a consequence of the equation $\mathcal{C}u_*=0$ which is
now $A_0u_*-bu_*+bPu_*=0$ (see also \cite{lasota03}).
\end{rem}


General sufficient conditions for existence of invariant densities for stochastic
operators have been summarized by \cite[Section 5]{almcmbk94} and
\cite{rudnickipichortyran02}.  We now discuss when a stochastic semigroup
$\{P(t)\}_{t\ge 0}$ is partially integral. Let $P$ be the transition operator
corresponding to a stochastic kernel $\mathcal{J}$. If there is a Borel measurable
function $p\colon
\X\times\X\to[0,\infty)$ such that 
\begin{equation*}
\int_{\X}\int_{\X}p(x,y)\varphi(y)dydx>0\quad \text{and}\quad \mathcal{J}(x,B)\ge
\int_{B}p(y,x)dy,\quad B\in\B(\X),
\end{equation*}
then $\{P(t)\}_{t\ge 0}$ is partially
integral~\cite{pichorrudnicki00,pichorrudnicki00b}.  Now, if $P$ is the
Frobenius-Perron operator corresponding to a nonsingular transformation $\sigma
\colon \X\to \X$ then $\mathcal{J}(x,B)=\mathbf{1}_B(\sigma(x))$, $x\in \X$,
$B\in\B(\X)$, and we have the following result.

%

\begin{prop}[\cite{pichorrudnicki00,pichorrudnicki00b}]\label{p:pichrud}
Let $\sigma\colon \X\to \X$ be continuously differentiable with $\sigma'(x)\neq 0$
for almost every $x\in \X$ and $P$ be the Frobenius-Perron operator corresponding
to $\sigma$. Assume that the semigroup $\{P(t)\}_{t\ge 0}$ is stochastic. If there
is $\bar{x}\in \X$ such that $\varphi$ is continuous at $\bar{x}$,
$\varphi(\bar{x})>0$,  and $g(\sigma(\bar{x}))\neq \sigma'(\bar{x})g(\bar{x})$
then $\{P(t)\}_{t\ge 0}$ is partially integral.
\end{prop}

\section{Specific examples}\label{sec:exp-growth}

As mentioned in the introduction, we were originally motivated by the work of
Lasota et al.~\cite{al-mcm-jt-92} who, in turn, were trying to understand the
experimentally observed asymptotic properties of cell property densities in
cellular populations in a growth phase. This led, naturally, to a consideration of
dynamics such that $g > 0$ for $x\in(d_0,d_1)$ as we have considered in
Sections~\ref{sec:St} and~\ref{sec:Pt}, and there are a number of concrete
situations in this category to which we can apply our results. This section
illustrates some of these.  Several of the examples are drawn from the field of
cell cycle kinetics, while others illustrate the necessity of certain assumptions.




Suppose first that the reset function is given by $\sigma(x)=x-\rho(x)$, where
$\rho$ is a continuously differentiable
function with $\rho'(x)<1$. 
Then the transition operator $P$ is of the form $Pu(x)=\lambda '(x
)u(\lambda(x))1_\X(\lambda(x))$, where $\lambda (x) =\sigma^{-1}(x)$ is the
inverse of $\sigma$, and  we have the evolution equation
\[
 \dfrac {\partial u (t,x)}{\partial t}
    =- \dfrac {\partial (g(x) u(t,x))}{\partial x} -\varphi (x) u(t,x)  +
    u (t ,\lambda (x)) \varphi ( \lambda (x)) \lambda ' (x )1_\X(\lambda(x)).
\]
If $Q(d_1)=\infty$ then the stochastic operator $K$ as defined in \eqref{d:opP} is
of the form
\begin{equation}\label{op:cc}
Ku(x)=\lambda'(x)\frac{\varphi(\lambda(x))}{g(\lambda(x))}\mathbf{1}_{(d_0,d_1)}(\lambda(x))\int
_{d_0}^{\lambda (x)} e^{Q(y)-Q(\lambda(x))}
    u(y)\, dy,
\end{equation}
which is the transition operator corresponding to \eqref{eq:an}, where
$T(a,\tau)=\sigma(Q^{-1}(Q(a)+\tau))$.
\begin{exmp}Lasota and Mackey \cite{almcm84} considered a very general cell
cycle model for the evolution of the distribution of `mitogen' at cell birth in
which $g$ was a $C^1$ function on $[0,2l)$,  such that $g(x)>0$ for $x>0$ and
$G(2l)=\infty$, where $l$ is finite or not. Further $\varphi$ was a continuous
function on $[0,2l)$ such that $\varphi (0)= 0$
\[
\liminf_{x\to 2l} \varphi(x)>0\quad \text{when } l <\infty \quad \text{and }
\quad\liminf_{x\to\infty} q(x)
> 0,
\]
where $q(x) = \varphi(x)/g(x)$. In their model $\lambda(x)=2x$. They were able to
show that successive generations had densities evolving under the action of a
stochastic operator
\[
Ku(x) = 2 q(2x) \int_0^{2x} \exp \left [ -\int _y^{2x} q(z) dz \right ] u(y) \,dy
\quad\text{for}\quad 0<x<l
\]
and that $K$ is asymptotically stable.
\end{exmp}
\begin{exmp}
Building on this model Mackey et al.~\cite{mcm-ss-86} took  $g(x) = x(2-x)/b$ and
$ \varphi(x) = S(x-1)\mathbf{1}_{(1,2)}(x)$ with $b,S > 0$ and $x \in (0,2)$ to
fit a number of {\it in vitro} cell cycle data sets. With these choices for $g$
and $\varphi$ it is straightforward to show that the unique stationary density of
mitogen is
    $$
    v_*(x) = Sb \cdot 2^{Sb} (x-\tfrac 12)\left [{x(1-x)} \right ]^{(Sb/2)-1} \quad \mbox{for} \quad x \in (\tfrac 12, 1).
    $$

\end{exmp}

Assume that $\X=(d_0,\infty)$ and $Q(\infty)=\infty$. We can
rewrite the operator $K$ given by \eqref{op:cc} in the form
\begin{equation}\label{op:cc1} Ku(x)=\int _{d_0}^{\lambda (x)} -\dfrac
{\partial}{\partial x}\left (
    e^{Q(y)-Q(\lambda(x))} \right )
    u(y)\, dy.
\end{equation}
Asymptotic properties of this operator have been well studied
\cite{gackilasota90,al-mcm-jt-92,baronlasota93,rudnicki95}.
\begin{prop}[\cite{gackilasota90}]\label{p:blas}
Assume that $\varphi(x)>0$ for $x>d_0$ and $Q(d_0)=0$. If
\begin{equation}\label{eq:blas}
\liminf_{x\to \infty}\bigl( Q(\lambda(x))-Q(x)\bigr)>1
\end{equation}
then $K$ as defined in \eqref{op:cc1} has a strictly positive invariant density
and if $Q(\lambda(x))-Q(x)\le 1$ for all $x>d_0$ it has no invariant density.
\end{prop}

\begin{rem}\label{r:blas}
The assumption $Q(d_0)=0$ can not be omitted in Proposition~\ref{p:blas}, as the
following example shows.

Let $\X=(0,\infty)$, $\lambda(x)=2x$, and $Q(x)=b\log x$. The operator $K $ is now
\[
K u(x)=\frac{b}{2^b x^{b+1}}\int_{0}^{2x}y^b u(y)\,dy.
\]
If we take $u(x)=1/x$ then $Ku(x)=u(x)$ for all $b>0$. This shows that $K $ has a
subinvariant function which is strictly positive and 
not integrable. Since $K $ overlaps supports, it has no invariant density
\cite[Remark 6]{rudnicki95} for any $b> 0$, but condition \eqref{eq:blas} holds
whenever $b\log 2>1$.
\end{rem}


\begin{exmp}\label{ex:l1}
Consider the following functions
\[
g(x)=k,\quad k>0, \quad  \varphi(x)=px^\alpha,\quad p>0, \alpha> -1, \quad
\lambda(x)=2x, \quad x>0.
\]
 We have $G(x)=kx$ and $Q(x)=bx^{\alpha+1}/(\alpha+1)$, where $b=p/k$. The
 operator~$K$ has a strictly positive invariant density, by Proposition~\ref{p:blas}.
 Thus the~semigroup $\{P(t)\}_{t\ge 0}$ is
 stochastic, by Theorem~\ref{thm:uniqa}, and it is partially integral, by Proposition~\ref{p:pichrud},
 since $g(x/2)\neq g(x)/2$ for all $x$.
The domain of the operator $A$ is
\[
\mathcal{D}(A)=\{u\in L^1[0,\infty):u(0)=0, u\in \mathrm{AC}, u'\in L^1,
\int_{0}^{\infty}x^\alpha |u(x)|dx<\infty\}.
\]
The stationary density $u_*\in\mathcal{D}(A)$ for the semigroup $\{P(t)\}_{t\ge
0}$ is a solution of the equation
\[
u'(x)=-bx^\alpha u(x)+2b(2x)^\alpha u(2x),\quad x>0, 
\]
and is given by
\[
u_*(x)=\sum_{n=0}^\infty c_n e^{-Q(2^nx)},\quad\text{where}\quad
c_n=\frac{2^{\alpha+1}}{1-2^{n(\alpha+1)}}c_{n-1},\quad n\ge1,
\]
and $c_0$ is a normalizing constant.  The stationary density $v_*$ for the
operator $K$ is given by $v_*(x)=cx^\alpha u_*(2x)$, where $c$ is a normalizing
constant.
\end{exmp}

We continue with the above example, but now we take $\alpha=-1$ and show that for
certain values of the parameter $b$ the semigroup is stochastic and for others it
is not. Observe that we have $Q(x)=b\log x$, thus the operator $K$ is the operator
from Remark~\ref{r:blas} and it is not asymptotically stable.

\begin{exmp}\label{ex:l2} Let the functions $g$ and $\lambda$ be as in Example~\ref{ex:l1}. Let
$Q(x)=b\log x$, where $b=p/k$. We have
\[
P(\varphi R_1) u(x)=\frac{b e^{-2 k x}}{2^b x^{b+1}}\int_{0}^{2x}e^{ k y}y^b
u(y)dy.
\]
If we take $u(x)=x^{\beta-1}e^{- k x}$ then $u\in L^1$ for $\beta>0$ and
\[
P(\varphi R_1) u(x)=\frac{b 2^\beta}{b+\beta}u(x)e^{- k x}. 
\]

Assume that $b\log 2<1$. Since we can find $\beta>0$ such that $b 2^\beta\le
b+\beta$, the operator $P(\varphi R_1)$ has a subinvariant strictly positive
density, which shows that the semigroup $\{P(t)\}_{t\ge 0}$ is stochastic, by
Theorem~\ref{thm:uniq}.

Assume now that $b\log 2>1$ and take $k=1$. If we go back to \eqref{eq:an} then
\[
a_n=\frac{1}{2}a_{n-1}e^{\tau_n/b}\quad \text{and}\quad
\mathbb{E}(a_n^\gamma)=\frac{\mathbb{E}(a_0^\gamma)}{2^{n\gamma}}\Bigl(\frac{b}{b-\gamma}\Bigr)^n.
\]
We can find $\gamma<1$ such that $b 2^{-\gamma}<b-\gamma$. Since $
t_n=a_n-a_0+\sum_{i=1}^n a_i$, this shows that
$\sup_{n}\mathbb{E}(t_n^\gamma)<\infty$, so that the process is defined only up to
a finite random time and $\{P(t)\}_{t\ge 0}$ can not be stochastic.
\end{exmp}

\begin{exmp} Tyson and Hannsgen \cite{tyson86} in their cell cycle model
 consider a special case of the model of \citep{almcm84} in which they let
$\X=(\sigma,\infty)$, where $\sigma<1$ and consider the following functions
\[
g(x)=kx,\quad \varphi(x)=\left\{
                           \begin{array}{ll}
                             0, & x<1 \\
                             p, & x\ge 1
                           \end{array}
                         \right.
,\quad \sigma(x)=\sigma x.
\]
They show that the unique steady state $v_*$ is given by
\[
v_*(x) = \dfrac {r-1}{\sigma} \left( \dfrac x \sigma \right )^{-r},
\]
where the exponent $r>1$ must satisfy
\[
b-(r-1) = b \sigma ^{r-1} \quad\text{and}\quad b \ln \frac{1}{\sigma}>1, \quad
\text{where }b=\frac{p}{k}.
\]
In this example we have $g(\sigma x)=\sigma g(x)$ for all $x$ and the semigroup
$\{P(t)\}_{t\ge 0}$ is not partially integral. Although it has a unique strictly
positive stationary density, it is not asymptotically stable due to a possible
synchronization \cite{diekmann84}.

\end{exmp}



We conclude this section with an example when the reset function $\sigma$ depends
additionally on an external variable.   
Let $\X=(0,\infty)$, $\Theta=(0,1)$, $\nu$ be a measure on $\Theta$ with a density
$h$, and the reset function $\sigma$ be of the form  $ \sigma(x,\theta)=x-\theta x
$. Then the transition operator $P$ has the kernel $p$
\[
p(x,y)=1_{(0,x)}(y)\psi\Bigl(\frac{x}{y}\Bigr)\frac{1}{y}, \quad 
\text{where } \psi(\theta)=h(1-\theta),
\]
and the evolution equation is
$$
\dfrac {\partial u (t,x)}{\partial t}=-\dfrac {\partial g(x) u(t,x) }{\partial
x}-\varphi(x) u(t,x)+ \int_{x}^\infty
\psi\Bigl(\frac{x}{y}\Bigr)\frac{\varphi(y)}{y}u(t,y)dy.
$$
The operator $K$ has the kernel
\begin{equation}\label{e:k1}
k(x,y)=\int_{\max\{x,y\}}^\infty\psi\Bigl(\frac{x}{z}\Bigr)\frac{\varphi(z)}{z
g(z)}e^{Q(y)-Q(z)}dz, \quad x,y\in (0,\infty).
\end{equation}

\begin{exmp}
Suppose that  $\varphi(x)/g(x)=b x^\alpha$ for all $x>0$, where $b>0$,
$\alpha>-1$. We have $ Q(x)=b x^{\alpha+1}/(\alpha+1)$. We provide the form of the
invariant density for $K$ when $\psi(z)=\beta z^{\beta-1}$ for $z\in[0,1]$ and
$\beta>0$. It is easily seen that the invariant density for the operator $K$ is of
the form
\[
v_*(x)=\frac{b^{\gamma}}{(\alpha+1)^{1+\gamma}\Gamma(\gamma)}x^{\beta-1}e^{-Q(x)},\quad
\gamma=\frac{\beta}{\alpha +1}.
\]We have
\[
R_0v_*(x)=\frac{xv_*(x)}{\beta g(x)}.
\]
If $R_0v_*\in L^1$ then $\{P(t)\}$ is an asymptotically stable stochastic
semigroup by Theorem~\ref{thm:uniqa}. For example when $\alpha=0$ then $v_*$ is
the gamma distribution, while if $\alpha=\beta=1$ then
\[
v_*(x)=\frac{\sqrt{b}}{2\sqrt{2\pi}}e^{-b x^2/2}.
\]
\end{exmp}

\section{Decay instead of growth}\label{sec:other}

We assumed in Sections~\ref{sec:St} and~\ref{sec:Pt} that
$g(x)>0$ for $x\in(d_0,d_1)$, and illustrated the applicability
of our results to concrete situations in Section
\ref{sec:exp-growth}.  In this section we discuss a situation
when instead of  growth there is degradation, so now we suppose
that $g(x)<0$ for $x\in(d_0,d_1)$. Results analogous to those
of Sections~\ref{sec:St} and~\ref{sec:Pt}
with similar proofs are valid in this case, and we illustrate the applicability of these
to models for the stochastic regulation of gene expression. 

Let the functions $G$, $Q$ be defined as in~\eqref{eq:GQ}.
Observe that now $G$ is decreasing and $Q$ is nonincreasing. If
$G(d_0)=+\infty$ then $\pi_t(\X)\subseteq \X$ for all $t\ge 0$
and $t\mapsto\pi_tx$ is decreasing. Let $\{P_0(t)\}_{t\ge 0}$
and $\{S(t)\}_{t\ge 0}$ be as in~\eqref{eq:dpo}
and~\eqref{eq:dst}. Thus, if $g<0$ and $G(d_0)=+\infty$ the
conclusions of Theorem~\ref{thm:sP0t} remains valid. The
analogue of Theorem~\ref{thm:ASt} with the same method of proof
reads as follows.

%
\begin{thm}
If $g<0$ and $G(d_0)=+\infty$ then the operator $A_0$ defined on the domain
\begin{equation}\label{d:dA0n}
\mathcal{D}(A_0)=\{u\in \SX: gu\in \ACloc,\; A_0 u\in \SX, \;\lim_{x\uparrow
d_1}g(x)u(x)=0\}
\end{equation}
is the generator of the semigroup $\{P_0(t)\}_{t\ge 0}$ and the operator $A$
defined on $\mathcal{D}(A)=\mathcal{D}(A_0)\cap \SXP$  is the generator of
$\{S(t)\}_{t\ge 0}$.

Moreover, the resolvent of the operator $A$ is of the form
\[
R(\lambda,A)v(x)=
\int_{x}^{d_1}\frac{1}{|g(x)|}e^{Q_\lambda(y)-Q_\lambda(x)}v(y)\;dy,\quad v\in
\SX.
\]
\end{thm}

The operator $R_0$ is now defined with the help of the function
\[
r_0(x,y)=\mathbf{1}_{(x,d_1)}(y)\frac{1}{|g(x)|} e^{Q(y)-Q(x)}
\]
and the assertions of Theorem~\ref{thm:uniqa} remain valid under the assumption
that $g<0$ and $G(d_0)=Q(d_0)=+\infty$.

In particular, if the operator $P$ has kernel $p$ then the operator $K$ has the
kernel $k$
\[
k(x,y)=\int_{d_0}^y p(x,z)\frac{\varphi(z)}{|g(z)|}e^{Q(y)-Q(z)}\,dz
\]
and is stochastic if and only if $Q(d_0)=\infty$.

Let $\X=\Theta=(0,\infty)$ and  $\sigma(x,\theta)=x+\theta$. 
Then the operator $P$ is the convolution operator with the measure $\nu $, i.e. if
$\zeta_1$ has density $u$ and $\theta_1$ has distribution $\nu $ then $P u$ is the
density of $\zeta_1+\eta_1$. Assume that $\nu$ has a density $h$. Then
\[
p(x,y)=\mathbf{1}_{(0,x)}(y) h(x-y),
\]
so that our evolution equation is
$$
\dfrac {\partial u (t,x)}{\partial t}=-\dfrac {\partial g(x) u(t,x) }{\partial
x}-\varphi(x) u(t,x)+ \int_{0}^xh(x-y)\varphi(y)u(t,y)dy.
$$
\begin{exmp} Friedman et al.~\cite{friedman06} have considered stochastic
aspects of gene expression following from bursts of protein production, and their
formulations are special cases of our results. Identifying their  $w$ and $c$ with
$w(x-y)dy=h(x-y)dy - \delta_{x} (dy)$ and $\varphi(x)=k_1 c(x)$
and taking $g(x)=-\gamma x$ then our equation 
becomes identical, in a steady state situation to their Equation 6
$$
-\dfrac {\partial \gamma x u_*(x) }{\partial x} = \int_{0}^x
\varphi(y)u_*(y)w(x-y) dy.
$$
Then, following \cite{friedman06}, let $h(y) =\frac1b e^{- y/b}$ be the
exponential distribution and let, in their first model,  $\varphi(x)=k_1$. Then
the equation, as they have shown, has as a solution the density of the gamma
distribution. In considering the second  model of \cite{friedman06} that treated
transcription factor regulation of its own transcription, we are let to consider
the following function
\[
\varphi(x)=k_1\frac{1}{1+x^\alpha}+k_1\epsilon.
\]
As they have shown the corresponding density is given by
$$
u_*(x) = c x^{a(1+\epsilon)-1} e^{-x/b} \left [ \frac{1}{1+x^\alpha}\right
]^{a/\alpha},
$$
where $c$ is a normalizing factor.

\end{exmp}

\section*{Acknowledgments} This work was supported by the Natural
Sciences and Engineering Research Council (NSERC, Canada), by the Mathematics of
Information Technology and Complex Systems (MITACS, Canada), and by the State
Committee for Scientific Research (Poland) Grant N N201 0211 33 (MT-K). This
research was partially carried out while MCM was visiting the Institut f{\"ur}
theoretische Physik, Universit{\"a}t Bremen, and while MT-K was visiting McGill
University.

%

\end{document}